\documentclass[10pt]{amsart}
\usepackage[english]{babel}
\usepackage{amssymb}
\usepackage{enumitem}
\usepackage{hyperref}
\usepackage[initials]{amsrefs}
\usepackage[all]{xy}
\usepackage{amsthm}
\SelectTips{cm}{}

\newcommand{\bbN}{{\mathbb N}}

\newcommand{\bbR}{{\mathbb R}}
\newcommand{\bbZ}{{\mathbb Z}}

\newcommand{\mfg}{\mathfrak{g}}
\newcommand{\mfh}{\mathfrak{h}}

\newcommand{\abs}[1]{{\left\lvert #1\right\rvert}}

\newcommand{\overto}[1]{{\buildrel{#1}\over\longrightarrow}}

\newcommand{\acts}{\curvearrowright}

\newcommand{\ssc}[2]{\frac{1}{#1}\bullet{#2}}
\newcommand{\scl}[2]{\operatorname{scl}_{#1}\left({#2}\right)}

\newtheorem{mthm}{Theorem}

\newtheorem{theorem}{Theorem}[section]
\newtheorem{lemma}[theorem]{Lemma}

\newtheorem{corollary}[theorem]{Corollary}

\newtheorem{proposition}[theorem]{Proposition}

\theoremstyle{definition}

\newtheorem{defn}[theorem]{Definition}

\newtheorem{example}[theorem]{Example}

\newtheorem{remarks}[theorem]{Remarks}

\begin{document}

\title[Equivariant Random Maps Between Nilpotent Groups ]{Differentiability of Integrable Measurable Cocycles Between Nilpotent Groups  }
\date{\today}

\author{Michael Cantrell}
\address{University of Illinois at Chicago, Chicago}
\email{mcantr2@uic.edu}

\begin{abstract}
We prove an analog for integrable measurable cocycles of Pansu's differentiation theorem for Lipschitz maps between Carnot-Carath\'eodory spaces. This yields an alternative, ergodic theoretic proof of Pansu's quasi-isometric rigidity theorem for nilpotent groups, answers a question of Tim Austin regarding integrable measure equivalence between nilpotent groups, and gives an independent proof and strengthening of Austin's result that integrable measure equivalent nilpotent groups have bi-Lipschitz asymptotic cones. Our main tools are a nilpotent-valued cocycle ergodic theorem and a Poincar\'e recurrence lemma for nilpotent groups.
\end{abstract}

\maketitle

\section{Introduction}
In \cite{PansuLip} Pansu proved the following seminal quasi-isometric rigidity theorem for nilpotent groups.

\begin{theorem}[Pansu \cite{PansuLip}]\label{t:Pansu1}
Finitely generated quasi-isometric nilpotent groups have isomorphic associated Carnot groups.
\end{theorem}

 He did this in two independently interesting steps. First, he identified the unique asymptotic cone of a finitely generated nilpotent group equipped with a left-invariant inner metric as an associated Carnot group with a Carnot-Carath\'eodory metric \cite{Pansu}. The second step is Pansu's differentiation theorem.
 
 \begin{theorem}[Pansu \cite{PansuLip}]\label{t:Pansu2}
 A bi-Lipschitz map between Carnot groups is differentiable almost everywhere. Moreover, the derivative induces a group isomorphism.
 \end{theorem}
 
 Since asymptotic cones of quasi-isometric groups are bi-Lipschitz, one deduces Theorem \ref{t:Pansu1}.

Measure equivalence (hereafter `ME') is an equivalence relation on groups introduced by Gromov \cite{Gromov2} that is a measure-theoretic parallel of quasi-isometry. It has been the object of considerable study: Furman's survey \cite{Furman} provides a thorough overview. However, a fundamental result of Ornstein and Weiss \cite{OW} implies that measure equivalence collapses all amenable groups into one equivalence class.

A measure equivalence between two groups implicitly defines a pair of measurable cocycles over probability measure preserving (pmp) actions of those groups. In their study of rigidity of hyperbolic lattices, \cite{BFS} Bader, Furman and Sauer have sharpened measure equivalence to a finer equivalence relation, called integrable measure equivalence (IME), by considering only those measure equivalences for which these cocycles satisfy an integrability condition.

Recently Austin and Bowen \cite{Austin} showed that the single ME class of infinite amenable groups splits into many IME classes.  Bowen showed that the growth type of a group is preserved by IME, and Austin used Bowen's result to prove the following.

\begin{theorem}[Austin \cite{Austin}]\label{t:Austin}
Finitely generated integrable measure equivalent nilpotent groups have bi-Lipschitz asymptotic cones.
\end{theorem}

 Notice that combining Theorems \ref{t:Pansu2} and \ref{t:Austin} one deduces the IME analog of Theorem \ref{t:Pansu1}.
 
 \begin{theorem}\label{t:IME-Carnot}
 Finitely generated integrable measure equivalent nilpotent groups have isomorphic associated Carnot groups.
 \end{theorem}
 
 However this proof is not entirely satisfying as it does not `see' the group isomorphism through the IME. In his proof, Austin considers the measurable cocycle as an equivariant family of random maps between the f.g. groups that induces a sequence of measurable maps $\kappa_{x,n}$ between the associated Carnot groups indexed by the rescaling $1/n$ in the asymptotic cone construction. He then proves that with high probability a subsequence of these maps converge to a bi-Lipschitz map between the Carnot groups. Austin then asks the natural question (Question 5.2 \cite{Austin}): Is there a bi-Lipschitz group isomorphism between the Carnot groups to which this sequence of random maps converge with high probability on bounded sets? We answer this question in the affirmative.

\begin{mthm}\label{t:derivative} \hfill{}\\
Suppose $\Gamma$ and $\Lambda$ are IME f.g. nilpotent groups with associated Carnot groups $G_\infty$ and $H_\infty$. Let $\kappa_{x,n}$ be the maps as in Question 5.2 \cite{Austin}. Then there is a bi-Lipschitz group isomorphism $\Phi:G_\infty\to H_\infty$ to which $\kappa_{x,n}$ converge on bounded sets with high probability as $n\to\infty$

\[
\kappa_{x,n}\longrightarrow \Phi.
\]
\end{mthm}

\begin{remarks} \label{rmk:A} \leavevmode
\begin{enumerate}
 \item  In \cite{Shalom} Shalom keenly observed that amongst f.g. amenable groups, quasi-isometry implies uniform measure equivalence, which in particular implies IME. Therefore Theorem \ref{t:derivative} implies Theorem \ref{t:Pansu1}. While we do not rely logically on Theorem \ref{t:Pansu2}, we do use the idea of the Pansu derivative. \label{rmk:Shalom}
 \vspace{3mm}
 \item  One might say that the isomorphism $\Phi$ is the Pansu derivative of the given measurable cocycle. Indeed, in the deterministic case $\Phi$ is the usual Pansu derivative. 
  \vspace{3mm}

 \item Theorem \ref{t:derivative} is for any Carnot-Carath\'eodory metrics on $G_\infty$ and $H_\infty$. All Carnot-Carath\'eodory metrics on a given Carnot group are bi-Lipschitz, so in what follows we may not specify the metric. Moreover $\Phi$ being a group isomorphism implies it is bi-Lipchitz.
 
 \end{enumerate}
\end{remarks}

Theorem \ref{t:derivative} is an immediate consequence of Theorem \ref{T:Main}, which has the spirit of a nilpotent-valued cocycle ergodic theorem.

\begin{mthm}\label{T:Main}\hfill{}\\
Let $\Gamma$,$\Lambda$ be f.g. IME nilpotent groups with associated cocycles $\alpha:\Gamma\times X\to\Lambda$ and $\beta:\Lambda\times Y\to\Gamma$, and let $G_\infty$ and $H_\infty$ be the associated Carnot groups of $\Gamma$ and $\Lambda$. Then there exists a bi-Lipschitz group isomorphism $\Phi:G_\infty\to H_\infty$ so that for all $g\in G_\infty$

\[
 (n,\gamma_n)\longrightarrow g \qquad \text{implies} \qquad (n,\alpha(\gamma_n,x))\longrightarrow \Phi(g) 
\]

where the convergence is in the sense of the asymptotic cone, and the second convergence is in measure. The same is true after exchanging the roles of $\Gamma$,$\Lambda$, $\alpha$,$\beta$, and $\Phi$,$\Phi^{-1}$.
\end{mthm}

See $\S$\ref{ss:graded lie algebra} for the definition of convergence in the asymptotic cone.

\begin{remarks} \leavevmode
\begin{enumerate}
\item Convergence in measure is the best one can hope for given the $L^1$ integrability assumption. To have pointwise convergence even in case $\Gamma=\Lambda=\bbZ^d$ one must assume $L^{d,1}$ (Lorentz-space) integrability. The correct integrability assumption for pointwise convergence of ergodic theorems for nilpotent groups is commonly believed to be related to the growth type of the group.
  \vspace{.01mm}
\item All of the theorems stated above are true for f.g. \emph{polynomial growth} groups, which by \cite{Gromov} are those groups with finite index nilpotent subgroups. Theorem \ref{T:Main} is insensitive to finite index and finite kernels, so we reduce to the torsion-free nilpotent case. See \S\ref{ss:reduction}.
\end{enumerate}
\end{remarks}

The proof of Theorem \ref{T:Main} is a natural extension of ideas developed in the author's thesis \cite{CF}. The idea is that, following Pansu \cite{Pansu}, the large scale geometry of f.g. nilpotent groups depends only on the behavior of the projection to abelianization. Therefore, to understand the large scale geometric behavior of a random map $\alpha(\cdot,x):\Gamma\to\Lambda$, we project it to the abelianization and integrate. Since a section of the abelianization generates the whole group, we can write all elements in terms of that section. We then use the cocycle identity to decompose arbitrary elements into a product of those coming from (a section of) the abelianization, which allows us to promote the cocycle ergodic theorem for cocycles with values in $\bbR^d$, which is easy,  to the desired cocycle ergodic theorem with values in $\Lambda$. 

We remark that while it is almost immediate that the limiting map $\Phi$ is a homomorphism, the nilpotent Poincar\'e recurrence Lemma \ref{l:Poincare} is needed to show that $\Phi$ has (the obvious candidate as) an inverse.

The rest of the paper is organized as follows. The second section sets notation, gathers background information regarding nilpotent groups, asymptotic cones, and measure equivalence, and reduces to the torsion-free nilpotent case. In section 3 we study asymptotics along iterates of a single element. In section 4 we combine the results of \S 3 with Lemma \ref{L:multiplication} to understand asymptotics along arbitrary elements. Finally in section 5 we define $\Phi$, prove Theorem \ref{T:Main}, and deduce Theorem \ref{t:derivative}.

On a first reading of this paper, one may wish to skip the proofs in subsection \ref{ss:nilgeom}, as the statements are intuitive. Also, one may wish to skip the proofs in section \ref{s:iterates}, which are the most technical part of the paper.

We conclude the introduction by noting that, in light of Remark \ref{rmk:A} \eqref{rmk:Shalom}, one might hope to develop a nilpotent IME rigidity theory parallel to that of quasi-isometry (\cite{Shalom}, \cite{Sauer}, \cite{KP}).

\subsection*{Acknowledgements}
I would like to express my sincere gratitude to Tim Austin and to my advisor Alex Furman.

\section{Background and Notation} \label{Background}

\subsection{Integrable Measure Equivalence}\label{ss:ime}

Two infinite discrete countable groups $\Gamma$, $\Lambda$ are \textbf{measure equivalent} if there exists an infinite measure space $(\Omega,m)$ with a measurable, measure preserving action of $\Gamma\times\Lambda$ so that the actions $A:\Gamma\curvearrowright (\Omega,m)$ and $B:\Lambda\curvearrowright (\Omega,m)$ admit finite measure fundamental domains $Y,X\subset \Omega$. The space $(\Omega,m)$ together with the $\Gamma\times\Lambda$ action is called a measurable coupling of $\Gamma$ and $\Lambda$. By restricting attention to an ergodic component, one may always assume that $m$ is ergodic for the $\Gamma\times\Lambda$ action.  

The fundamental domains $Y$ and $X$ for the $G$ and $H$ actions give rise to functions

\[
\alpha:\Gamma\times X\longrightarrow \Lambda \qquad \text{and} \qquad \beta:\Lambda\times Y\longrightarrow \Gamma
\]

defined uniquely by requiring 

\[
B(\lambda) y\in A(\beta(\lambda,y)^{-1}) Y \qquad \text{and} \qquad A(g) x \in B(\alpha(\gamma,x)^{-1}) X \quad~ \forall x\in X~ \forall y\in Y.
\]

There are associated finite measure preserving actions $\Gamma\curvearrowright (X,m|_X)$ and $\Lambda\curvearrowright (Y,m|_Y)$ (whose actions we denote by $\cdot$) defined by requiring that

\[
A(\gamma)x=B(\alpha(\gamma,x)^{-1})(\gamma\cdot x)\qquad \text{and} \qquad  B(\lambda)y=A(\beta(\lambda,y)^{-1}) (\lambda \cdot y).
\]

If $m$ is ergodic for $\Gamma\times\Lambda$ then the actions $\Gamma\curvearrowright (X,m|_X)$ and $\Lambda\curvearrowright (Y,m|_Y)$ are ergodic. We may assume after renormalizing that both $m|_X$ and $m|_Y$ are probability measures. Finally, $\alpha$ and $\beta$ are \textbf{measurable cocycles} over the pmp actions in the sense that

\begin{equation}\label{eq:cocycles}
\alpha(\gamma_1\gamma_2,x)=\alpha(\gamma_1,\gamma_2\cdot x)\alpha(\gamma_2,x) \qquad \text{and} \qquad \beta(\lambda_1\lambda_2,y)=\beta(\lambda_1,\lambda_2\cdot y)\beta(\lambda_2,y)
\end{equation}

for all $\gamma_1,\gamma_2\in\Gamma$, $\lambda_1,\lambda_2\in\Lambda$ and for $m$ a.e. $x\in X$, $y\in Y$. Most of our reasoning will be about these cocycles. 

Replacing the fundamental domain $Y$ with one of its $H$ translates only translates the cocycle $\beta$. Since countably many translates of $Y$ cover $\Omega$, we may therefore assume that $m(X\cap Y)>0$. Moreover, (see \cite{Austin} for more details) if 

\[
x\in X\cap Y \cap \gamma^{-1}(X\cap Y)
\]

then

\[
\beta(\alpha(\gamma,x),x)=\gamma.
\]

Given finitely generated groups $\Gamma,\Lambda$, a cocycle $\alpha:\Gamma\times X\to\Lambda$ over a pmp action $\Gamma\curvearrowright (X,\mu)$ is \textbf{integrable} if, for some (any) choice of finite generating set for $\Lambda$

\[
\|\abs{\alpha(\gamma,\cdot)}_\Lambda\|_1=\int_X \abs{\alpha(\gamma,x)}_\Lambda d\mu(x)<\infty \qquad \forall \gamma\in\Gamma
\]

where $\abs{\cdot}_\Lambda$ is the word norm associated to the generating set. The subadditivity of $\abs{\cdot}_\Lambda$ implies

\[
\|\abs{\alpha(\gamma,\cdot)}_\Lambda\|_1\le \abs{\gamma}_\Gamma \cdot \max_{s\in S}\|\abs{\alpha(s,\cdot)}_\Lambda\|_1
\]

where $\abs{\cdot}_\Gamma$ is any word norm associated to a finite generating set for $\Gamma$. 

Finally, finitely generated groups $\Gamma$ and $\Lambda$ are \textbf{integrably measure equivalent} if they admit a measurable coupling so that the associated cocycles \eqref{eq:cocycles} are integrable. This is an equivalence relation independent of choice of generating sets. For more details, see \cite{Furman}.

Recall that measurable events $E_n \subset (X,m)$ occur with high probability (whp) if $m(E_n)\to 1$ as $n\to\infty$. We say that a sequence of measurable functions $f_n:X\to [0,\infty)$ is $o(n)$ in probability (or `whp') if for all $\epsilon>0$ one has $m(f_n(x)/n<\epsilon)\to 1$ as $n\to\infty$. Thus for example $d_\Lambda(\alpha(\gamma^n,x),\lambda)=o(n)$ in probability means that for all $\epsilon,\delta>0$ there is $N$ so that for all $n\ge N$ one has $m(d_\Lambda(\alpha(\gamma^n,x),\lambda)<n\delta)<\epsilon$. Similarly for $O(n)$.

\subsection{The Associated Graded Lie Algebra}\label{ss:graded lie algebra}
Let $\Gamma$ be a finitely generated torsion free nilpotent group. Recall that by a theorem of Mal'cev \cite{Malcev} there is a unique connected, simply connected nilpotent Lie group $G$, called the Mal'cev completion of $\Gamma$, in which $\Gamma$ embeds as a (necessarily cocompact) lattice.

Since $G$ is simply-connected, the exponential map $\exp:\mfg:=Lie(g)\to G$ from the Lie algebra of $G$ to $G$ is a diffeomorphism, so we can work with the Lie algebra.

Let $\mfg$ be the Lie algebra of $G$, and set
\[
	\mfg^1:=\mfg, \qquad \mfg^{i+1}:=[\mfg,\mfg^{i}].
\]
Being nilpotent, $G$ satisfies $\mfg^{r+1}=\{0\}$ for some $r \in \bbN$. 
Since $[\mfg^i,\mfg^j]\subset \mfg^{i+j}$ the Lie bracket on $\mfg$ defines a bilinear map 
\[
	\left(\mfg^{i}/\mfg^{i+1}\right)\otimes \left(\mfg^{j}/\mfg^{j+1}\right)\ \overto{}\ (\mfg^{i+j}/\mfg^{i+j+1}),
\]
which can then be used to define the Lie bracket $[-,-]_\infty$ on 
\begin{equation}\label{e:frakv-decomp}
	\mfg_{\infty}:=\bigoplus_{i=1}^{r} \frak{v}_i,\qquad\textrm{where}\qquad \frak{v}_i:=\mfg^i/\mfg^{i+1}
\end{equation}
by extending the above maps linearly.

The resulting pair $(\mfg_\infty,[-,-]_\infty)$ is called the \textbf{graded Lie algebra} associated with $\mfg$.
Note that the linear maps
\[
	\delta_t:\mfg_\infty \to \mfg_\infty,\qquad \delta_t(v_1,\dots,v_{r})=(t\cdot v_1,t^2\cdot v_2,\dots,t^{r}\cdot v_{r}),
\]
satisfy $\delta_t([v,w]_\infty)=[\delta_t(v),\delta_t(w)]_\infty$ and $\delta_{ts}=\delta_t\circ \delta_s$ 
for $v,w\in\mfg_\infty$, $t,s>0$. Hence $\{\delta_t \mid t>0\}$ is a one-parameter family of automorphisms of the Lie algebra $\mfg_\infty$, and therefore define a one-parameter family of automorphisms of the Lie group $G_\infty:=\exp_\infty(\mfg_\infty)$,
that we will still denote by $\{\delta_t \mid t>0\}$. (Here we denote the exponential map $\mfg_\infty\to G_\infty$
by $\exp_\infty$ to distinguish it from $\exp:\mfg\to G$).

Choose a splitting of $\mfg$ as a direct sum of vector subspaces 
\begin{equation}\label{e:V-decomp}
	\mfg=V_1\oplus\cdots\oplus V_{r},\qquad\textrm{so\ that}\qquad
	\mfg^i=V_i\oplus\cdots\oplus V_{r},
\end{equation} 

and choose a vector space identification
$L:\mfg\to\mfg_\infty$ so that $L(V_i)=\frak{v}_i$ the $i$th summand of $\mfg_\infty$.
For $t>0$ define the vector space automorphism $\delta_t$ of $\mfg$ 
by $\delta_t(v)=t^i\cdot v$ for $v\in V_i$ ($i=1,\dots,r$). Note that $\{\delta_t\mid t>0\}$ are \textit{not}  Lie algebra automorphisms of $\mfg$ in general. Nevertheless they induce maps $\{\delta_t\mid t>0\}$ from $G$ to $G$ which we still denote $\delta_t$. Note also that the maps $\delta_t$ defined on $\mfg$ and on $\mfg_\infty$ are conjugate through $L$.

Now the Lie bracket $[-,-]_t$ on $\mfg$, given by 
\[
	[v,w]_t:=\delta_{\frac{1}{t}}\left([\delta_t(v),\delta_t(w)]\right),
\]
defines a Lie algebra structure on $\mfg$ that is isomorphic to the original $[-,-]=[-,-]_1$ via $\delta_t$.

However, one has
\[
	[L(v),L(w)]_\infty=\lim_{t\to\infty} [v,w]_t
\]
due to the fact that for $v\in V_i$, $w\in V_j$ the "leading term" of $[v,w]$ lies in $V_{i+j}$,
while the higher terms that belong to $V_{i+j+1}\oplus\cdots\oplus V_{r}$ 
become insignificant under the rescaling (see \cite{Pansu}).  
Using the $\log:G\to \mfg$ and $\exp_\infty:\mfg_\infty\to G_\infty$ maps 
we obtain a family of maps

\begin{equation}\label{e:sclt}
	\scl{t}{-}:\Gamma\ \overto{<}\ G\ \overto{\log}\ \mfg\ \overto{\delta_{t^{-1}}}\ 
	\mfg\ \overto{L}\ \mfg_\infty\ \overto{\exp_\infty}\ G_\infty\qquad (t>0)
\end{equation}
that explains the asymptotic cone description of Pansu \cite{Pansu} as follows.

Let $d$ be an inner left-invariant metric $d$ on $\Gamma$ and  
\[
	(\Gamma,\frac{1}{t}\cdot d,e)\ \overto{GH}\ (G_\infty,d_\infty,e)
\]
the Gromov-Hausdorff convergence. Then a sequence $\gamma_i\in \Gamma$, rescaled by $t_i^{-1}$
with $t_i\to\infty$ as $i\to\infty$, 
converges to $g\in G_\infty$ iff $\scl{t_i}{\gamma_i}\to g$ in $G_\infty$.

We shall often write 
\[
	g=\lim_{i\to\infty} \ssc{t_i}{\gamma_i}\qquad\textrm{instead\ of}\qquad \scl{t_i}{\gamma_i}\to g.
\]
The metric part of the statement shows that for $t_i\to\infty$ and $\gamma_i,\gamma'_i\in \Gamma$
\begin{equation}\label{e:lengthscaling}
	g=\lim_{i\to\infty}\ \ssc{t_i}{\gamma_i},\quad g'=\lim_{i\to\infty}\ \ssc{t_i}{\gamma'_i}
	\qquad
	\Longrightarrow
	\qquad
	d_\infty(g,g')=\lim_{i\to\infty} \frac{1}{t_i}\cdot d(\gamma_i,\gamma'_i).
\end{equation}
The limiting distance $d_\infty$ on $G_\infty$ is \textbf{homogeneous} in the sense that
\[
	d_\infty(\delta_s(g),\delta_s(g'))=s\cdot d_\infty(g,g')\qquad (g,g'\in G_\infty,\ s>0).
\]	
This distance is left-invariant (this follows from Lemma~\ref{L:multiplication}). The distance $d_\infty$ arises from the sub-Finsler Carnot-Carath\'eodory construction.

\begin{lemma}\label{L:multiplication}\hfill{}\\
	Given sequences $t_i\to\infty$, $\gamma_i,\gamma'_i\in\Gamma$ with $\ssc{t_i}{\gamma_i}\to g$ and $\ssc{t_i}{\gamma'_i}\to g'$
	then $\ssc{t_i}{\gamma_i\gamma'_i}\to gg'$.
\end{lemma}
\begin{proof}
	This follows from the Baker-Campbell-Hausdorff formula (cf. \S 3.3 and the proof of Lemma~5.5 in \cite{Bre}).
\end{proof}

\subsection{Logarithmic Coordinates}
We will use the so called \textit{logarithmic coordinates} throughout this paper, which are described as follows. Choose a real basis $\{X_1,\ldots ,X_m\}$ for $\mfg$ that respects the decomposition \eqref{e:V-decomp}. When we write $g=(x_1,\ldots ,x_m)\in G$ we mean that $g=\exp(x_1X_1+\cdots x_mX_m)$. These are the \textbf{logarithmic coordinates} of $G$. Thus if $g=(x_1,\ldots ,x_m)$ and $h=(x_1',\ldots ,x_m')$ then the product $gh=(y_1,\ldots ,y_m)$ where 
\[
\exp(x_1X_1+\cdots +x_mX_m)\exp(x_1'X_1+\cdots x_1'X_m)=\exp(y_1X_1+\cdots y_mX_m).
\] 

 In light of the vector space isomorphism $L:\mfg\to\mfg_\infty$ the basis for $\mfg$ yields a basis for $\mfg_\infty$ that respects the decomposition \eqref{e:frakv-decomp}. Throughout this paper we will think of $\mfg$ and $\mfg_\infty$ as occupying the same real vector space, only with different Lie brackets $[-,-]$ and $[-,-]_\infty$. We also use the logarithmic coordinates for $\mfg_\infty$, the only difference in definition being the Lie bracket.

Let $d=\dim V_1$. Then there exist constants $\eta_{i}\in\bbN$, $d<i \le m$ so that $\Gamma$ embeds in $G$ in logarithmic coordinates as

\[
\Gamma=\left \{ {(a_1,\ldots ,a_d,\eta_{d+1}a_{d+1},\ldots ,\eta_{m}a_m):a_i\in\bbZ}\right\}<G.
\]

Thus we have identified $\Gamma<G\equiv \bbR^m\equiv G_\infty$. Therefore we think of $\Gamma<G$ and $G_\infty$ as occupying the same copy of $\bbR^m$. We denote the group product in $\Gamma<G$ by $g\cdot h$ or simply by $gh$, and the group product in $G_\infty$ by $g\star h$. We will always denote a word norm on a discrete nilpotent group $\Gamma$ or $\Lambda$ by $\abs{\cdot}_\Gamma$ or $\abs{\cdot}_\Lambda$, a word norm on a nilpotent Lie group $G$ or $H$ by $\abs{\cdot}_G$ or $\abs{\cdot}_H$ and a Carnot-Carath\'eodory norm on a graded nilpotent Lie group $G_\infty$ or $H_\infty$ by $\abs{\cdot}_\infty$, and their associated metrics $d_\Gamma$, $d_\Lambda$, $d_H$, $d_H$, and $d_\infty$. Thus we can without notational ambiguity omit the linear identification $L:G\equiv G_\infty$. For example if $\gamma,\sigma\in\Gamma$ then $\abs{\gamma}_\infty$ means unambiguously $\abs{L\gamma}_\infty$ and $\gamma \star \sigma$ means $L\gamma \star L\sigma$.

Since $V_1\cong \mfg/\mfg^2$, the sets

\[
\{ (x_1,\ldots ,x_d,0,\ldots ,0)\in G\}\cong\bbR^d \qquad \text{and} \qquad \{(a_1,\ldots ,a_d,0,\ldots ,0)\in\Gamma\}\cong\bbZ^d
\]

are complete sets of coset representatives for $G/G^2$, $G_\infty/G_\infty^2$ and (the torsion-free part of) $\Gamma/\Gamma^2$. We will use these choices of coset representatives in the arguments that follow. We define the projections on to the abelian and commutator coordinates for $\Gamma$, $G$, and $G_\infty$ by

\begin{align*}\label{eq:projections}
\pi_{ab}(a_1,\ldots ,a_m)&=(a_1,\ldots ,a_d,0,\ldots ,0) \\
\pi_{com}(a_1,\ldots ,a_m)&=(0,\ldots ,0,a_{d+1},\ldots ,a_m).
\end{align*}

\subsection{Some Nilpotent Geometry}\label{ss:nilgeom}

We now collect some basic nilpotent geometry facts. We make no claim to originality in this subsection.

We will use the following Lemma of Guivarc'h repeatedly throughout this paper to simplify our arguments. Since asymptotic statements are not sensitive to quasi-isometry, the Guivarc'h Lemma allows us to prove asymptotic statements for only one of $(H,d_H)$ or $(H_\infty,d_{\infty})$.

\begin{lemma}[Guivarc'h \cite{Gui}; see also \cite{Bre} Theorem 3.7]\label{l:Gui}
Let $K$ be a compact neighborhood of the identity in a simply connected nilpotent Lie group $G$ and $d_G(g,h)=\inf\{n\ge 1: g^{-1}h\in K^n\}$. Then for any homogeneous quasi-norm $\abs{\cdot}$ on $G$ there is a constant $C>0$ so that

\[
\frac{1}{C}\abs{g}\le d_G(e,g)\le C\abs{g}+C.
\]
\end{lemma}

We now use the Guivarc'h Lemma to give succinct proofs of several nilpotent geometric facts, which could also be proved by induction on nilpotency class. All of the statements are true independent of choice of symmetric generating set, but we work with a fixed generating set $S$ with associated norm $\abs{\cdot}_\Gamma$ and metric $d_\Gamma$ to be concise. All constants depend on $\Gamma$ and $S$. Let us say that two functions $f,g:\Gamma\to\bbR_+$ are \textit{quasi-isometric} if there exists $C>0$ so that for all $\gamma\in\Gamma$, $f(\gamma)/C-C\le g(\gamma)\le Cg(\gamma)+C$. The following lemma is a natural statement regarding the asymptotic word growth of each coordinate in a nilpotent group. Define, for each $1\le i \le m$, the \textit{degree} $d_i=\deg(X_i)$ to be the greatest $j$ so that $X_i\in \mfg^{j-1}$.

\begin{lemma} \label{F:degree}
 For each $1\le i \le m$ there exist $0<c_1<c_2<\infty$ so that for all $n\in\bbZ$ 
 
 \begin{equation*}
 c_1 n^{1/d_i}\le \abs{(0,\ldots ,n,\ldots 0)}_\Gamma \le c_2n^{1/d_i}
 \end{equation*}
 
 where the non zero term is in the $i$-th coordinate.
 
Moreover, if $[X_{i_1},\cdots ,[X_{i_{l-1}},X_{i_l}]\cdots ]=cX_t$ where $i_r\in \left \{ {1,\ldots ,m}\right \}$ and $c\neq 0$, then 

\[
\sum_{r=1}^{l} d_{i_r} \le d_t.
\]
\end{lemma}

\begin{proof}
The following is a quasi-norm on $G$

\[
\abs{(x_1,\ldots ,x_m)}_m:=\max_i \abs{x_i}^{1/d_i}.
\]

$(G,\abs{\cdot}_G)$ restricted to $\Gamma$ is quasi-isometric to $(\Gamma,\abs{\cdot}_\Gamma)$, while by the Guivarc'h Lemma, $(G,\abs{\cdot}_G)$ is quasi-isometric to $(G,\abs{\cdot}_m)$. But $\abs{(0,\ldots ,n,\ldots ,0)}_m=n^{1/d_i}$. Since $\Gamma$ is discrete we may absorb the additive factors.
The moreover statement is obvious from the definitions. 
\end{proof}

\begin{lemma}\label{l:asymp coords}
For each $1\le i \le m$ set
\begin{align*}
f_i(n)&=\abs{(0,\ldots ,0,n,0,\ldots ,0)}_\Gamma \\
g_i(n)&=\min_{a_j} \abs{(a_1,\ldots ,a_{i-1},n,a_{i+1},\ldots ,a_m)}_\Gamma.
\end{align*}

where the non-zero coordinate is in the $i$-th coordinate. Then there exists $1\le C<\infty$ so that for all $n\in\bbN$

\[
f_i(n)\le cg_i(n).
\]
\end{lemma}

\begin{proof}
\begin{align*}
f_i(n) &\le cn^{1/d_i} \le c\min_{a_j}\abs{(a_1,\ldots ,a_{i-1},n,a_{i+1},\ldots ,a_m)}_m \\
&\le cc_1\min_{a_j}\abs{(a_1,\ldots ,a_{i-1},n,a_{i+1},\ldots ,a_m)}_\Gamma+c_2  \\
&\le (cc_1+c_2)\min_{a_j}\abs{(a_1,\ldots ,a_{i-1},n,a_{i+1},\ldots ,a_m)}_\Gamma \qquad (n\neq 0)
\end{align*}
where we have used Lemma \ref{F:degree} and the Lemma of Guivarc'h.

\end{proof}

The next lemma says that projecting to the commutator coordinates only reduces word norm by a universal multiplicative constant. 
\begin{lemma}\label{l: compare com}
There is a constant $C>0$ so that $\forall \gamma\in\Gamma$

\[
\abs{\gamma}_\Gamma \ge C \abs{\pi_{com}\gamma}_\Gamma.
\]
\end{lemma}

\begin{proof}
 For all $\gamma\in\Gamma$ we have trivially

\[
\abs{\gamma}_m\ge \abs{\pi_{com}\gamma}_m.
\]

The Guivarc'h Lemma and the discreteness of $\Gamma$ finish the proof.
\end{proof}

\begin{lemma} \label{l:n grows}
There exists $l>0$ so that for all $\gamma\in\Gamma-\Gamma^2$ and for all $n$ $\abs{\gamma^n}_\Gamma\ge ln$.
\end{lemma}
\begin{proof}
If $\gamma\notin\Gamma^2$ then $\abs{\gamma^n}_m \ge n$. The Guivarc'h Lemma and the discreteness of $\Gamma$ finish the proof.
\end{proof}

\begin{lemma} \label{l:o(n)}
The functions $\abs{\cdot}_\Gamma, \abs{\cdot}_G, \abs{\cdot}_m,\abs{\cdot}_\infty:\Gamma\to\bbR_+$ are all quasi-isometric. Moreover,
\begin{align*}
\abs{\gamma_n}_\Gamma&=\abs{(a_{n,1},\ldots ,\ldots ,a_{n,m})}_\Gamma =o(n)&  &\iff& \abs{a_{n,t}}&=o(n^{d(t)})& \qquad &\forall 1\le t\le m  \\
\abs{\gamma_n}_\Gamma&=\abs{(a_{n,1},\ldots ,\ldots ,a_{n,m})}_\Gamma=O(n)&  &\iff& \abs{a_{n,t}}&=O(n^{d(t)})& \qquad &\forall 1\le t\le m\\
\abs{g_n}_G&=\abs{(a_{n,1},\ldots ,\ldots ,a_{n,m})}_G=o(n)&  &\iff& \abs{a_{n,t}}&=o(n^{d(t)})& \qquad &\forall 1\le t\le m \\
\abs{g_n}_G&=\abs{(a_{n,1},\ldots ,\ldots ,a_{n,m})}_G=O(n)&  &\iff& \abs{a_{n,t}}&=O(n^{d(t)})& \qquad &\forall 1\le t\le m \\
\abs{g_n}_\infty&=\abs{(a_{n,1},\ldots ,\ldots ,a_{n,m})}_\infty=o(n)&  &\iff& \abs{a_{n,t}}&=o(n^{d(t)})& \qquad &\forall 1\le t\le m \\
\abs{g_n}_\infty&=\abs{(a_{n,1},\ldots ,\ldots ,a_{n,m})}_\infty=O(n)&  &\iff& \abs{a_{n,t}}&=O(n^{d(t)})& \qquad &\forall 1\le t\le m \\
\end{align*}

where $a_{n,j}\in\bbZ$ ($a_{n,j}\in\bbR$) is the $j$-th coordinate of $\gamma_n\in\Gamma$ ($g_n\in G$).
\end{lemma}

Note that the stronger statement that the corresponding left-invariant metrics $(G,d_G)$ and $(G_\infty,d_\infty)$ are quasi-isometric is not true in general.

\begin{proof}
For the first statement, recall that $(G,\abs{\cdot}_G)$ restricted to $\Gamma$ is quasi-isometric to $(\Gamma,\abs{\cdot}_\Gamma)$, while by the Guivarc'h Lemma, $(G,\abs{\cdot}_G)$ is quasi-isometric to $(G,\abs{\cdot}_m)$. Note that $(G,\abs{\cdot}_m)$ and $(G_\infty,\abs{\cdot}_m)$ are equal (under the implicit linear identification $L$), and that $(G_\infty, \abs{\cdot}_\infty)$ is a quasi-norm. Since any two quasi-norms on the same group are bi-Lipschitz, we have proven the first statement.

The moreover statement follows from the first statement together with the fact that 

\[
\abs{g_n}_m=o(n) \iff \abs{a_{n,t}}=o(n^{d(t)}) \qquad \forall 1\le t \le m
\]

and similarly for $O(n)$.
\end{proof}

\begin{lemma}\label{l:oO(n)}
If $g_n\in G_\infty$ is a sequence such that 

\begin{enumerate}
\item $\abs{\pi_{com}g_n}_\infty=o(n)$ \label{eq:lambda o(n)}
\item $\abs{\pi_{ab}g_n}_\infty=O(n)$ \label{eq:lambda O(n)}
\end{enumerate}
 then 

\[
\abs{(\pi_{ab}g_n)^{-1}\star g_n}_\infty=o(n).
\]

\end{lemma}

\begin{proof}
Let $g_n=(a_{n,1},\ldots ,a_{n,m})$ so that $\pi_{ab}g_n=(a_{n,1},\ldots ,a_{n,d},0,\ldots ,0)$ and \\ $(\pi_{ab}g_n)^{-1}=(-a_{n,1},\ldots ,-a_{n,d},0,\ldots ,0)$. Using the Baker-Campbell-Hausdorff formula, the nilpotency of $G_\infty$ and linearity of the bracket

\[
(\pi_{ab}g_n)^{-1}\star g_n=a_{n,d+1}X_{d+1}+\cdots +a_{n,m}X_m+ h.o.t.
\]

where $h.o.t.$ are precisely the terms involving at least one bracket in the product 
\begin{equation}\label{eq:hot}
(-a_{n,1}X_1 - \cdots -a_{n,d}X_d)\star (a_{n,d+1}X_{d+1}+\cdots +a_{n,m}X_m).
\end{equation}

Since the abelian coordinates of $(\pi_{ab}g_n)^{-1}\star g_n$ are all zero, by Lemma \ref{l:o(n)} it suffices to show that the $t$-th coordinate of $(\pi_{ab}g_n)^{-1}\star g_n$ is $o(n^{d_t})$ for every $d < t \le m$. By assumption $\abs{\pi_{com}g_n}_\infty=o(n)$, so it suffices to show that the contributions from \eqref{eq:hot} to each $t$ coordinate are $o(n^{d(t)})$, for $d<t\le m$. Using Baker-Campbell-Hausdorff again, for fixed $d<t\le m$ the contribution is a sum of finitely many terms of the form

\begin{equation*}
c[a_{n,i_1}X_{i_1},\cdots ,[a_{n,i_{l-1}}X_{i_{l-1}},a_{n,i_l}X_{i_l}]\cdots ]
\end{equation*}

where $c$ is a constant from the Baker-Campbell-Hausdorff formula, $i_1,\ldots i_l \in \left \{ {1,\ldots ,m}\right \}$ and for at least one $r$, $i_r\in\left\{ {d+1,\ldots ,m}\right \}$. Since the number of such terms depends only on $G_\infty$, it suffices to show that

\[
a_{n,i_1}\cdots a_{n,i_l}=o(n^{d_t}),
\]

which follows immediately from the fact that at least one $i_r\in\left\{ {d+1,\ldots ,m}\right \}$ and 
\begin{itemize}
\item  $\abs{a_{n,i_r}}=O(n)$ ~~~~~for $1\le i_r \le d$
\item  $\abs{a_{n,i_r}}=o(n^{d_{i_r}})$ ~~for $d<i_r\le m$ 
\item $\sum_{r=1}^l d_{i_r} \le d_t$.

\end{itemize}

\end{proof}

\begin{lemma}\label{l:commutator growth}
Let $g_n,h_n\in G$. If $\abs{g_n}_G=o(n)$ and $\abs{h_n}_G=O(n)$ then \\ $\abs{g_n^{-1}h_n^{-1}g_nh_n}_G=o(n)$. Moreover the same is true of any Carnot-Carath\'eodory norm on $G_\infty$ instead of $G$.
\end{lemma}

\begin{proof}
Let $g_n=(a_{n,1},\ldots ,a_{n,m})$, $h_n=(b_{n,1},\ldots ,b_{n,m})$ and suppose $\abs{g_n}_G=o(n)$ and $\abs{h_n}_G=O(n)$. By Lemma \ref{l:o(n)} 

\begin{align*}
\abs{a_{n,t}}&=o(n^{d_t})  \qquad~1\le t \le m \\
\abs{b_{n,t}}&=O(n^{d_t})  \qquad 1\le t \le m.
\end{align*}

Say $g_n=\exp v_n=\exp(a_{n,1}X_1+\cdots +a_{n,m}X_m)$ and $h_n=\exp w_n=\exp(b_{n,1}X_1+\cdots +b_{n,m}X_m)$, so

\begin{equation}\label{eq:commutator exp}
g_n^{-1}h_n^{-1}g_nh_n=\exp -v_n\exp -w_n \exp v_n \exp w_n=\exp([v_n,w_n]+\cdots )
\end{equation}

where the dots stand for terms involving three or more brackets. Let us examine the coefficient $c_r$ of $X_r$ in \eqref{eq:commutator exp}; it is a sum of finitely many terms of the form 

\begin{equation*}
ca_{n,i_1}\cdots a_{n,i_s}b_{n,j_1}\cdots b_{n,j_t} \qquad \text{where} \quad \sum_{\substack{1\le p \le s \\ 1\le q \le t}}d_{i_p}+d_{j_q}\le d_r
\end{equation*}

where $c$ is a (possibly zero) constant depending only on $G$, and $s\neq 0$, i.e. there is at least one $a_{n,i}$ term. Employing Lemma \ref{l:o(n)} again it suffices to show that each of these possible coefficients is $o(n^{d_r})$. Indeed there is a constant $c$ (coming from the $O(n^{d(j_q)})$) so that for all $\epsilon>0$ and all sufficiently large $n$ 

\begin{equation*}
\abs{a_{n,i_1}\cdots a_{n,i_s}b_{n,j_1}\cdots b_{n,j_t}}\le \epsilon n^{\sum d_{i_p}} c n^{\sum d_{i_q}} \le c\epsilon n^{d_r}.
\end{equation*}

To see that the same is true for $G_\infty$ with a Carnot-Carath\'eodory norm $\abs{\cdot}_\infty$, note that the proof only used Lemma \ref{l:o(n)} and nilpotency. 
\end{proof}

\subsection{Notation}

All of the above was true of a general finitely generated torsion-free nilpotent group $\Gamma$, though of course the groups $G$, $G_\infty$, as well as the corresponding dimension of the abelianization $d=\dim(G/[G,G])$, the nilpotency step $s$ and the vector space dimension $m$ all depend on $\Gamma$.

Let us fix two finitely generated torsion-free nilpotent groups $\Gamma$ and $\Lambda$ that are integrably measure equivalent with integrable cocycles as in \eqref{eq:cocycles} for which the action $\Gamma\acts (X,m)$ is pmp ergodic. We denote their Mal'cev Lie groups $G$ and $H$ and their graded lie groups $G_\infty$ and $H_\infty$, respectively. Let us now fix finite generating sets $S$ and $T$ for $\Gamma$ and $\Lambda$ respectively. We will denote their respective word norms $\abs{\cdot}_\Gamma$ and $\abs{\cdot}_\Lambda$ and the metrics $d_\Gamma$ and $d_\Lambda$. Let us also fix a compact generating set $K\subset H$ and denote the corresponding word norm and metric $\abs{\cdot}_H$ and $d_H$. Finally, there are the unique Carnot-Carath\'eodory metrics on $H_\infty$ and $G_\infty$ associated to $d_\Gamma$ and $d_\Lambda$ by \cite{Pansu}. Let us denote both by $d_\infty$, as no confusion can arise.

Keep in mind that, since we are not assuming Pansu's Theorem \ref{t:Pansu2} a priori we do not know whether $G_\infty$ and $H_\infty$ are isomorphic groups or that the dimensions of their abelianization are the same. So let us say that in logarithmic coordinates 

\begin{align*}
\Lambda<H\equiv H_\infty&\equiv \bbR^m  &\dim(\mfh/\mfh^2)&=d \\
\Gamma<G\equiv G_\infty &\equiv \bbR^{m'}   &\dim(\mfg/\mfg^2)&=d'.
\end{align*}

We will only work in the Lie algebras $\mfh$ and $\mfh_\infty$ of $H$ and $H_\infty$. Let us identify as in \eqref{e:V-decomp}

\[
V=V_1\oplus \dots \oplus V_s =\mfh=\mfh_\infty
\]

with Lie brackets $[-,-]_H$ and $[-,-]_\infty$. The projections we will use are for $\Lambda$, $H$ and $H_\infty$:

\begin{align*}
\pi_{ab}(a_1,\ldots ,a_m)&=(a_1,\ldots ,a_d,0,\ldots ,0) \\
\pi_{com}(a_1,\ldots ,a_m)&=(0,\ldots ,0,a_{d+1},\ldots ,a_m).
\end{align*}

 We will think of the image $\pi_{ab}(H)\cong\bbR^d$ in order to integrate, but for notational ease we suppress the identification. Now we may define two maps essential to what follows

\begin{align*}
&\alpha_{ab}:\Gamma\times X\to H \qquad &\alpha_{ab}(\gamma,x)=\pi_{ab} \circ \alpha(\gamma,x) \\
&\overline{\alpha_{ab}}:\Gamma\to H \qquad &\overline{\alpha_{ab}}(\gamma)=\int_X \alpha_{ab}(\gamma,x)dm(x).
\end{align*}

\subsection{Reduction to torsion-free nilpotent groups} \label{ss:reduction}
Here we reduce Theorem \ref{T:Main} to the case of torsion-free nilpotent groups. Finitely generated polynomial growth groups have finite index nilpotent subgroups, which themselves have finite normal torsion subgroups. Let $\Gamma'<\Gamma$ be a finite index subgroup. The action $\Gamma'\acts (X,m)$ has at most $[\Gamma:\Gamma']$-many ergodic components permuted by the $\Gamma$ action. Let $\tau_1,\ldots ,\tau_l\in\Gamma$ be a complete set of representatives for $\Gamma' \backslash \Gamma$. Consider an ergodic component $X'$ and the integrable cocycle $\alpha':\Gamma'\times X'\to \Lambda$ obtained by restriction. Suppose $\ssc{n}{\gamma_n} \to g \in G_\infty$. For each $n$ write $\gamma_n=\gamma'_n \tau_{n_i}$ where $\tau_{n_i} \in \{ \tau_1,\ldots ,\tau_l \}$ and $\gamma'_n\in\Gamma'$. Then $\ssc{n}{\gamma'_n}\to g$ so by Theorem \ref{T:Main} $\ssc{n}{\alpha(\gamma'_n,x)}\to \Phi(g)$ for some $\Phi$ that a priori depends on the ergodic component $X'$. Now the cocycle equality
$\alpha(\gamma_n,x)=\alpha(\gamma'_n\tau_{n_i},x)=\alpha(\gamma'_n,\tau_{n_i} x)\alpha(\tau_{n_i},x)$ implies

\[
d_\Lambda(\alpha(\gamma_n,x),\alpha(\gamma'_n,\tau_{n_i}x))=\abs{\alpha(\tau_{n_i},x)}_\Lambda
\]

which is bounded by a constant independent of $n$ with high probability by Markov's inequality. Therefore

\[
d_\infty(\ssc{n}{\alpha(\gamma'_n,x)},\ssc{n}{\alpha(\gamma_n,x)})=o(n) \qquad \text{whp}.
\]

Now let $N$ be a finite normal subgroup of $\Gamma$. Then $\Gamma/N$ acts ergodically by pmp transformations on $(X,m)/N$. Since $N$ is finite, we can find a measurable section $s:X/N\to X$ of $\pi: X\to X/N$. For every $x\in X$, there is $n_x\in N$ so that $n_x \cdot s\pi(x)=x$. Define

\[
f:X \to \Lambda \qquad f(x)=\alpha(n_x,s\pi(x))
\]

and the cocycle cohomologous to $\alpha$ via $f$

\[
\alpha^f(\gamma,x)=f(\gamma x)^{-1}\alpha(\gamma,x)f(x).
\]

Notice that $f$ takes finitely many values, so $\alpha^f$ is integrable. A direction computation shows that $\alpha^f$ restricted to $N$ is the trivial map, so $\alpha^f$ descends to a cocycle

\[
\alpha^f:\Gamma/N\times X/N \to \Lambda.
\]

Finally, if $\gamma_n\in\Gamma$ is such that $\ssc{n}{\gamma_n}\to g$, then also $\ssc{n}{\overline{\gamma_n}}\to g$ where $\overline{\gamma}=\gamma N\in \Gamma/N$. Thus $\ssc{n}{\alpha^f(\overline{\gamma_n},\pi x)}\to \Phi(g)$. Again since $f$ takes finitely many values, another application of the Markov inequality shows that 

\[
d_\Lambda(\alpha^f(\overline{\gamma_n},x),\alpha(\gamma_n,x))=o(n) \quad \text{whp}
\]

which finishes the proof.


\section{Asymptotic Behavior Along Iterates}\label{s:iterates}
In this section we analyze the asymptotic behavior of $\alpha(\gamma^n,x)$ as $n\to\infty$ for a given $\gamma\in\Gamma$. In the following section, we use the cocycle equation and the results of this section to understand the asymptotic behavior of an arbitrary $\alpha(\gamma,x)$. The idea in this section is to use the cocycle equation to see that $\alpha(\gamma^n,x)$ typically behaves like a homomorphism in to a nilpotent group. Crucially, one parameter families of elements in nilpotent groups experience an asymptotic decay in the higher order terms (commutator coordinates). In this section we use ergodicity to extend this phenomenon to a cocycle. Moreover, the position in the abelian coordinates stabilizes asymptotically, so that we have a perfect picture of the asymptotics of iterates: the higher order terms vanish, and the abelian coordinates tend to their average value.


The main result of this section is the following proposition.


%


 \begin{proposition}\label{p: iterates H_infty}
 For every $\gamma\in\Gamma$ 
 
 \[
\ssc{n}{\alpha(\gamma^n,x)}\longrightarrow \overline{\alpha_{ab}}(\gamma) \qquad \text{in probability}.
\]
 
 Equivalently,
 
 \[
d_{\infty}(\alpha(\gamma^n,x),\delta_n \overline{\alpha_{ab}}(\gamma))=o(n) \qquad \text{in probability}.
\]
 \end{proposition}

We will prove Proposition \ref{p: iterates H_infty} by analyzing the abelian and commutator coordinates separately.

\subsection{Abelianization Direction}
In this subsection we prove the following lemma describing the asymptotic behavior of $\alpha$ along iterates in the abelianization.

%
%

\begin{lemma}\label{l:abel Lambda}
For a.e. $x\in X$ and every $\gamma\in\Gamma$

\[
\frac{1}{n}\alpha_{ab}(\gamma^n,x)\to \overline{\alpha_{ab}}(\gamma)
\]

where the convergence is of vectors in $\bbR^d$.
\end{lemma}
The proof of the lemma is an easy application of the following found in the more general sub-additive case in \cite{Austin} and \cite{CF}.

\begin{proposition}\label{p:one dim}
Suppose $c:\Gamma\times X\to\bbR$ is a measurable cocycle over $\Gamma\curvearrowright (X,m)$ which is pmp ergodic. Then for a.e. $x\in X$ and every $\gamma\in\Gamma$

\[
\frac{1}{n}c(\gamma^n,x)\to\int_X c(\gamma,x)dm(x).
\]
\end{proposition}

\begin{proof}[Proof of Lemma \ref{l:abel Lambda}]
$\alpha_{ab}$ is itself a cocycle taking values in $\bbR^d$ which we can decompose as $d$ independent cocycles with values in $\bbR$. Indeed there are cocycles $\alpha_i:\Gamma\times X\to\bbR$ for $1\le i\le d$ so that

\[
\alpha_{ab}(\gamma,x)=(\alpha_1(\gamma,x),\ldots ,\alpha_d(\gamma,x),0,\ldots ,0).
\]

We can similarly decompose the averages

\[
 \overline{\alpha_{ab}}(\gamma)=(\int_X \alpha_1(\gamma,x)dm(x),\ldots ,\int_X\alpha_d(\gamma,x),0,\ldots ,0).
\]
Applying Proposition \ref{p:one dim} to each of the $\alpha_i$ finishes the proof.
\end{proof}

\subsection{Commutator Direction}
The purpose of this subsection is to prove the following lemma describing the asymptotic behavior of $\alpha$ along iterates in the commutator direction. 


\begin{lemma}\label{l:comm lambda}
For every $\gamma\in \Gamma$ 
\[
\abs{\pi_{com} \circ\alpha(\gamma^n,x)}_\Lambda=o(n) \qquad \text{in probability}.
\]

Moreover the same is true replacing the norm $\abs{\cdot}_\Lambda$ with $\abs{\cdot}_\infty$.
\end{lemma}

The moreover statement follows immediately from Lemma \ref{l:o(n)}. The proof of the main statement requires some preparation. The idea is to use the cocycle equation to write $\alpha(\gamma^{nk},x)=\alpha(\gamma^n,x_1)\alpha(\gamma^n,x_2)\cdots\alpha(\gamma^n,x_k)$ where $x_{i+1}=\gamma^{ni}x$. Using Lemma \ref{l:abel Lambda} whp the abelianization of each of the $\alpha(\gamma^n,x_i)\approx nv$ for some $v$, so that the commutator of $\alpha(\gamma^{nk},x)$ is roughly the sum of the commutators of the $\alpha(\gamma^n,x_i)$. This allows us to promote a linear bound on the commutator to an $o(n)$ bound since the commutator direction `should' grow at least quadratically. 

To begin, we use a weakened form of Proposition 3.2 from \cite{Austin} to obtain the $O(n)$ bound. Recall that given a pmp action $\Gamma \curvearrowright (X,m)$ a map $c:\Gamma\times X\to \bbR_+$ is a subadditive cocycle if 

\[
c(\gamma_1\gamma_2,x)\le c(\gamma_1,\gamma_2\cdot x)+c(\gamma_2,x) \qquad \forall \gamma_1,\gamma_2\in\Gamma~~m-a.e. x\in X.
\]

\begin{proposition}\label{p: z growth}
Given a subadditive cocycle $c:\Gamma\times X\to \bbR_+$, there is $M\ge 1$ such that for any $\epsilon>0$ there is $C=C(\epsilon)$ such that

\[
\abs{\gamma}_\Gamma \ge C\quad \implies \quad m(\abs{c(\gamma,x)}\ge M\abs{\gamma}_\Gamma)<\epsilon.
\]
\end{proposition}

We would like to use Proposition \ref{p: z growth} to draw conclusions about the size of the commutator of $\alpha(\gamma,x)$. To do this, we use Lemma \ref{l: compare com} which says that projection to the commutator increases word norm by at most a universal multiplicative constant, and Lemma \ref{l:n grows} which says that the norm of iterates of an element with nontrivial abelianization grows linearly up to a multiplicative constant. Combining this with Proposition \ref{p: z growth} we easily deduce the following $O(n)$ bound on the commutator growth. Since the word length of iterates of $\gamma\in\Gamma^2$ does \textit{not} grow linearly, we must deal with this easy case separately.

\begin{lemma}\label{l:M'}
For every $\gamma\in\Gamma-\Gamma^2$ there is $M'\ge 1$ so that for any $\epsilon>0$ there is $N$ so that for all $n\ge N$

\[
  m(\abs{\pi_{com}\circ\alpha(\gamma^n,x)}_\Lambda>M'n)<\epsilon.
\]
\end{lemma}

\begin{proof}
We apply Proposition \ref{p: z growth} to the subadditive cocycle $c:\Gamma\times X\to [0,\infty)$ defined by $c(\gamma,x)=\abs{\alpha(\gamma,x)}_\Lambda$. We obtain $M$ and set $M'=M\abs{\gamma}_\Gamma/k$ where $k$ is from Lemma \ref{l: compare com}. Fix $\epsilon>0$. Then there is $C$ so that

\[
\abs{\gamma}_\Gamma\ge C \quad \implies \quad m(\abs{\alpha(\gamma,x)}_\Lambda\ge M\abs{\gamma}_\Gamma)<\epsilon.
\]

Set $N=C/l$ where $l$ is from Lemma \ref{l:n grows}. Then since $\abs{\gamma^n}_\Gamma\le n\abs{\gamma}_\Gamma$,

\[
n\ge N \implies \abs{\gamma^n}_\Gamma \ge C \implies \quad m(\abs{\alpha(\gamma^n,x)}_\Lambda\ge Mn\abs{\gamma}_\Gamma)<\epsilon.
\]

Finally, by Lemma \ref{l: compare com}

\[
n\ge N \implies m(\abs{\pi_{com}\alpha(\gamma^n,x)}_\Lambda\ge M'n)<\epsilon.
\]

\end{proof}

The proof of Lemma \ref{l:comm lambda} is easy in case $\gamma\in\Gamma^2$.
\begin{lemma}
If $\gamma\in\Gamma^2$ then 
\[
\abs{\pi_{com} \circ\alpha(\gamma^n,x)}_\Lambda=o(n) \qquad \text{in probability}.
\]
\end{lemma}

\begin{proof}
By Markov's inequality there is $\kappa=\max_{s\in S} \|\abs{\alpha(s,\cdot)}_\Lambda\|_1$ so that for every $M\in\bbN$

\[
m(\abs{\alpha(\gamma^n,x)}_\Lambda>M\kappa\abs{\gamma^n}_\Gamma)<1/M.
\]

For $\gamma\in\Gamma^2$ there is a constant $c>0$ so that for all $n\in\bbN$ we have $\abs{\gamma^n}_\Gamma\le c\sqrt{n}$ (Lemma \ref{F:degree} and Lemma \ref{l:o(n)}). Thus for such $\gamma$ we have $\abs{\alpha(\gamma^n,x)}_\Lambda=o(n)$ whp. Lemma \ref{l: compare com} completes the proof.
\end{proof}

We need one more lemma before we can prove Lemma \ref{l:comm lambda}. Let us illustrate the idea behind the lemma through the example of the Heisenberg group. Recall that in logarithmic coordinates, the multiplication in the Heisenberg group is 
\[
(x,y,z)(x',y',z')=(x+x',y+y',z+z'+1/2(xy'-x'y)).
\]

The non-linear growth in the $z$-coordinate is given by the area enclosed by the triangle formed by $(x,y),(x+x',y+y')$ and $(0,0)$. So, if a pair of elements have very similar abelianizations, the $z$-coordinate of their product is approximately the sum $z+z'$. Now suppose we have $k$ elements with uniformly controlled $z$-coordinates and very similar abelianizations. Then the $z$-coordinate of their product grows approximately linearly. Thus the $z$-coordinate is $o(k)$ since the $z$-coordinate `should' grow quadratically. The following lemma generalizes this idea to general finitely generated torsion-free nilpotent groups.

We define the projection on to the first $t$ commutator coordinates

\[
\pi_t:\Lambda\to\Lambda \qquad \pi_t(a_1,\ldots ,a_m)=(0,\ldots ,0,a_{d+1},\ldots ,a_t,0,\ldots 0).
\]

Let $d_1$ be the $l^1$ metric on $\bbR^d$ and $\abs{\cdot}_1$ be the $l^1$ norm, so that $\abs{(x_1,\ldots ,x_d)}_1=\abs{x_1}+\cdots +\abs{x_d}$. 

\begin{lemma} \label{l:t}
Fix $0<M<\infty$ and $v\in\bbR^d$. For each $d \le t < m$ for all $\delta>0$ there exists $K\in\bbN$ and $\delta'>0$ so that for all $k\ge K$ and $\eta\ge 1$, whenever there exist $\lambda_1,\ldots ,\lambda_k\in\Lambda$ such that 

\begin{align} \label{eq:close v}
d_{1}(\pi_{ab}\lambda_i,v)&<\eta\delta'\abs{v}_1 \\
\abs{\pi_t\lambda_i}_\Lambda&<\eta\delta' \label{eq:t small} \\
\abs{\pi_{com}\lambda_i}_\Lambda&<\eta M\ \label{eq:t+1 small}
\end{align}

then 

\[
\abs{\pi_{t+1}\lambda_1\cdots\lambda_k}_\Lambda<\eta k\delta.
\]

\end{lemma}

\begin{proof}[Proof of Lemma \ref{l:comm lambda}]
Fix $\gamma\in\Gamma-\Gamma^2$. We obtain $M$ as in Lemma \ref{l:M'} and set $v=\overline{\alpha_{ab}}(\gamma)$. We prove by induction that for every $d\le t \le m$

\[
\abs{\pi_t \alpha(\gamma^n,x)}_\Lambda=o(n) \qquad \text{in probability}.
\]

For $t=d$ there is nothing to show. Suppose the result is known for $t$. Fix $\epsilon>0$ and $\delta>0$.  We apply Lemma \ref{l:t} with the given $\delta$, $M$ and $v$ to obtain $k=K$ and $\delta'$. Let $N$ be as in Lemma \ref{l:M'} applied to $\epsilon/k$, so that for all $\eta \ge N$ we have with probability at least $1-\epsilon/k$

\[
\abs{\pi_{com}\alpha(\gamma^\eta,x)}_\Lambda<\eta M.
\]

By taking $N$ larger if necessary, applying the inductive hypothesis to $\delta'/3$ and $\epsilon/k$ we obtain $N$ so that for all $\eta \ge N$ with probability at least $1-\epsilon/k$ we have

\[
\abs{\pi_t\alpha(\gamma^\eta,x)}_\Lambda<\eta\delta'/3.
\]

By taking $N$ larger again if necessary, by Lemma \ref{l:abel Lambda} for all $\eta \ge N$ with probability at least $1-\epsilon/k$ 

\[
d_{1}(\pi_{ab}\alpha(\gamma^{\eta},x),\eta v)<\eta\delta'\abs{v}/3.
\]

Since the $\Gamma$ action on $(X,\mu)$ is measure preserving, the previous three statements remain true if we replace any instance of $x$ with $g x$ for any $g\in\Gamma$. 

Finally, let $N$ be larger if necessary so that $k\le \delta'N$. Now let $p \ge kN$. Write $p=\eta k+r$ where $0\le r <k$ and $\eta\ge N$. Using the cocycle equation

\[
\alpha(\gamma^{k\eta+r},x)=\alpha(\gamma^\eta,x)\alpha(\gamma^{\eta},\gamma^{\eta} x)\cdots \alpha(\gamma^{\eta},\gamma^{(k-2)\eta},x)\alpha(\gamma^{\eta+r},\gamma^{(k-1)\eta},x).
\]

Since $\eta,\eta+r\ge N$, with probability at least $1-3\epsilon$ we have simultaneously for all $0\le i \le k-2$

\begin{align*}
 \abs{\pi_{com}\alpha(\gamma^\eta,\gamma^{i\eta}x)}_\Lambda&<\eta M \\
 \abs{\pi_t\alpha(\gamma^{\eta},\gamma^{i\eta}x)}_\Lambda&<\eta\delta'/3 \\
d_{1}(\pi_{ab}\alpha(\gamma^{\eta},\gamma^{i\eta}x),\eta v)&<\eta\delta'\abs{v}_1/3 
\end{align*}

and 

\begin{align*}
 \abs{\pi_{com}\alpha(\gamma^{\eta+r},\gamma^{(k-1)\eta}x)}_\Lambda&<(\eta+r) M \\
 \abs{\pi_t\alpha(\gamma^{\eta+r},\gamma^{(k-1)\eta}x)}_\Lambda&<(\eta+r)\delta'/3\\
d_{1}(\pi_{ab}\alpha(\gamma^{\eta+r},\gamma^{(k-1)\eta}x),(\eta+r) v)&<(\eta+r)\delta'\abs{v}_1/3.
\end{align*}

Since $r\le \delta'\eta$ the final three inequalities imply

\begin{align*}
 \abs{\pi_{com}\alpha(\gamma^{\eta+r},\gamma^{(k-1)\eta}x)}_\Lambda&<2\eta M \\
 \abs{\pi_t\alpha(\gamma^{\eta+r},\gamma^{(k-1)\eta}x)}_\Lambda&<\eta\delta' \\
d_{1}(\pi_{ab}\alpha(\gamma^{\eta+r},\gamma^{(k-1)\eta}x),\eta v)&<\delta'\abs{v}_1 
\end{align*}

where for the final inequality we have used the triangle inequality with intermediate term $(\eta+r)v$.

Therefore with probability at least $1-3\epsilon$ we apply Lemma \ref{l:t} and obtain

\[
\abs{\pi_{t+1}\alpha(\gamma^p,x)}_\Lambda<k\eta\delta<p\delta.
\]
\end{proof}

\begin{proof}[Proof of Lemma \ref{l:t}]
Fix $0<M<\infty$, $v\in\bbR^d$, $d\le t<m$, $\delta>0$ and $1>\delta'>0$. We will show in the proof how to choose $\delta'$ as a function of $\delta,\abs{v}_1,t$. Choose $K$ large so that $M/\sqrt{K} \le \delta^2$, and fix $k\ge K$ and $\eta\ge 1$. Suppose we have $\lambda_1,\ldots ,\lambda_k$ satisfying conditions \eqref{eq:close v} \eqref{eq:t small} \eqref{eq:t+1 small}.  Let us denote $\lambda_i=(a_{i,1},a_{i,2},\ldots ,a_{i,m})$  for each $1\le i \le k$, keeping in mind that only $a_{i,1}, \ldots ,a_{i,t+1}$ are relevant. Throughout this proof $c$ will denote an ever-changing constant this is independent of $\delta$, $\delta'$ and $\eta$.

We are concerned with the absolute value of the $t+1$ coordinate of the product $\lambda_1\cdots\lambda_k$. By Lemma \ref{F:degree} it suffices to show that the absolute value of this coordinate is at most $c(\eta k\delta)^{d(t+1)}$. The estimate we seek will follow from the Baker-Campbell-Hausdorff equation and the following constraints on the $a_{i,j}$ implied by conditions  \eqref{eq:close v}, \eqref{eq:t small} and \eqref{eq:t+1 small}:

\begin{align} \label{eq:aij}
\abs{a_{i,j}-a_{i',j}}\le c\eta\delta' \abs{v}_1& \qquad &\forall 1\le i,i'\le k \qquad &\forall 1\le j \le d \\
\abs{a_{i,j}}\le c(\eta\delta')^{d(j)}& \qquad &\forall 1\le i \le k \qquad &\forall d< j \le t \label{eq:aij 2}\\
\abs{a_{i,t+1}}\le c(\eta M)^{d(t+1)} & \qquad &\forall 1\le i \le k \qquad &{} \label{eq:aij 3} \\
\abs{a_{i,j}}\le c\eta^{d(j)} & \qquad &\forall 1\le i \le k \qquad & \forall 1\le j\le t \label{eq:aij 4}.
\end{align}

Indeed, setting $v=(v_1,\ldots ,v_d)$, from \eqref{eq:close v} we have $\sum_{j=1}^d \abs{a_{i,j}-v_j} \le \eta\delta'\abs{v}_1$ which implies in particular $\abs{a_{i,j}-v_j}\le \eta\delta'\abs{v}_1$ for all $i$, giving \eqref{eq:aij}. Combining \eqref{eq:t small}, Lemma \ref{l:asymp coords} and Lemma \ref{F:degree} we immediately arrive at \eqref{eq:aij 2}. Similarly combining \eqref{eq:t+1 small}, Lemma \ref{l:asymp coords} and Lemma \ref{F:degree} we arrive at \eqref{eq:aij 3}. It only remains to prove \eqref{eq:aij 4} in the case $1\le j \le d$, which follows from $\abs{a_{i,j}-v_j}\le \eta\delta'\abs{v}_1$ above and $\abs{v_j}\le \abs{v}_1$.

By the Baker-Campbell-Hausdorff equation we can express the product $\lambda_1\cdots \lambda_k$ as a sum of terms of the form 

\begin{equation}\label{eq:one term}
c[\lambda_{i_1}, \ldots ,[\lambda_{i_{l-1}},\lambda_{i_l}],\ldots ]
\end{equation}

where $i_j\in \left \{ {1,\ldots ,k}\right \}$ for each $1\le j\le l\le m$. We emphasize that it is possible that the indices are repeated, i.e. that $i_j=i_{j'}$ while $j\neq j'$. We are only interested in the brackets that contribute to the coefficient of $X_{t+1}$.  We replace each $\lambda_i$ with $\sum_{j=1}^m a_{i,j}X_j$ in each of the summands \eqref{eq:one term} above. Using linearity of the Lie bracket, the result is a sum of terms of the form

\begin{equation}\label{eq:one term aij}
c[a_{i_1,j(i_1)}X_{j(i_1)},\ldots ,[a_{i_{l-1},j(i_{l-1})}X_{j(i_{l-1})},a_{i_l,j(i_l)}X_{j(i_l)}]\ldots ] 
\end{equation}

where for each $i_r$ we have chosen $j(i_r)\in \left \{ {1,\ldots ,{t+1}}\right \}$. By Lemma \ref{F:degree}, we have that 

\begin{equation}\label{eq:j constraint}
\sum_{r=1}^l d_{j(i_r)} \le d_{t+1}
\end{equation}

so that in particular $l\le t+1$. We will show that each such term is small by analyzing the possibilities for the choices $j(i_r)$ above. We consider three cases.

For the first case we consider all terms with $j(i_r)=t+1$ for some $r$. Note that in this case, in view of \eqref{eq:j constraint} in fact \eqref{eq:one term aij} becomes 

\[
ca_{i_1,t+1}X_{t+1}.
\]

In view of \eqref{eq:aij 3}, summing these over all $1\le i_1 \le k$, the total contribution to the $t+1$ term from this case is, in absolute value, at most

\[
ck(\eta M)^{d_{t+1}}\le c n^{d_{t+1}} \delta^{d_{t+1}}k^{1+d_{t+1}/2}
\]

by our choice of $k$. This suffices since we may assume $d_{t+1}\ge 2$.

For the second case, we consider all terms in which at least one of the $j(i_r) \in \left \{ {d,\ldots ,t+1}\right \}$. By linearity we pull out all of the constants $a_{i,j}$ and consider the size of their product. By our assumption and \eqref{eq:aij 2} one of the terms is at most $c(\eta\delta')^{d_{j(i_r)}}$ and by \eqref{eq:aij 4} the rest of the terms are at most $c\eta^{d_{j(i_r)}}$. Therefore their product is at most 
\[
c\delta' \eta^{\sum_{r=1}^l d_{j(i_r)}}\le c\delta'\eta^{d_{t+1}}.
\]

Since there are finitely many such terms independent of $\delta'$, by taking $\delta'$ small as a function of $\delta,c,t$ and the number of such terms, the total contribution to the $t+1$ coordinate of the product $\lambda_1\cdots \lambda_k$ from terms of the second type is as desired.

For the third and final case we group each term into pairs and use antisymmetry, as follows.  We may assume all terms $i(j_r)\in \left \{ {1,\ldots ,d} \right \}$. In particular the inner most term $[a_{i_{l-1},j(i_{l-1})}X_{j(i_{l-1})},a_{i_l,j(i_l)}X_{j(i_l)}]$ has $j(i_{l-1})=s, j(i_l)=t$ for some $s,t \in [1,\ldots d]$. We pair the terms for which $j(i_{l-1})=s, j(i_l)=t$ with that for which $j(i_{l-1})=t, j(i_l)=s$, and all other $j(i_r)$ equal. By anti-symmetry of the bracket, the sum of these two terms is 

\begin{align*}
[a_{i_1,j(i_1)}X_{j(i_1)},\ldots ,[a_{i_{l-1},j(i_{l-1})}X_{j(i_{l-1})},a_{i_l,j(i_l)}X_{j(i_l)}]\ldots ] \\
+[a_{i_1,j(i_1)}X_{j(i_1)},\ldots ,[a_{i_{l-1},j(i_l)}X_{j(i_l)},a_{i_l,j(i_{l-1})}X_{j(i_{l-1})}]\ldots ] \\
=[a_{i_1,j(i_1)}X_{j(i_1)},\ldots ,(a_{i_{l-1},j(i_{l-1})}a_{i_l,j(i_l)}-a_{i_{l-1},j(i_l)}a_{i_l,j(i_{l-1})})[X_{j(i_{l-1})},X_{j(i_l)}]\ldots ]
\end{align*}

Pulling the constants out and considering the absolute value of the coefficient, we are concerned with the absolute value of 

\begin{equation}\label{eq:area}
a_{i_1,j(i_1)}\cdots a_{i_{l-2},j(i_{l-2})}(a_{i_{l-1},j(i_{l-1})}a_{i_l,j(i_l)}-a_{i_{l-1},j(i_l)}a_{i_l,j(i_{l-1})}).
\end{equation}

By properties \eqref{eq:aij} and \eqref{eq:aij 4} and the triangle inequality we have

\begin{align*}
\abs{(a_{i_{l-1},j(i_{l-1})}a_{i_l,j(i_l)}-a_{i_{l-1},j(i_l)}a_{i_l,j(i_{l-1})})}&\le \abs{a_{i_{l-1},j(i_{l-1})}a_{i_l,j(i_l)}-a_{i_l,j(i_l)}a_{i_l,j(i_{l-1})}} \\
&+\abs{a_{i_l,j(i_l)}a_{i_l,j(i_{l-1})}-a_{i_{l-1},j(i_l)}a_{i_l,j(i_{l-1})}}  \\
&\le\abs{a_{i_l,j(i_l)}}\abs{a_{i_{l-1},j(i_{l-1})}-a_{i_l,j(i_{l-1})}} \\
&+\abs{a_{i_l,j(i_{l-1})}}\abs{a_{i_l,j(i_l)}-a_{i_{l-1},j(i_l)}} \\
&\le c\eta c\eta\delta'\abs{v}_1+c\eta c\eta\delta'\abs{v}_1=c\eta^2\delta'\abs{v}_1.
\end{align*}

Now by \eqref{eq:aij 4} each of the other terms in the product \eqref{eq:area} has absolute value at most $c\eta$. Putting this together with the preceding and noting that $l\le d_{t+1}$, the absolute value of \eqref{eq:area} is at most $c\eta^{d_{t+1}}\delta'\abs{v}_1$. Since there are a finite number of such terms independent of $\delta'$, by taking $\delta'$ small as a function of $\delta, c,\abs{v}_1$, the total contribution to the absolute value of the $t+1$ coordinate of $\lambda_1\cdots \lambda_k$ from terms from the third case is as desired. This finishes the proof.

\end{proof}

%

%


\subsection{Proof of Proposition \ref{p: iterates H_infty}}
Finally we can combine Lemmas \ref{l:oO(n)}, \ref{l:abel Lambda} and \ref{l:comm lambda} to prove Proposition \ref{p: iterates H_infty}.

\begin{proof}[Proof of Proposition \ref{p: iterates H_infty}]
Fix $\gamma\in\Gamma$. Chow's Theorem and Lemma \ref{l:abel Lambda} imply

\begin{equation}\label{eq:2nd sum}
d_\infty(\alpha_{ab}(\gamma^n,x), n\overline{\alpha_{ab}}(\gamma))=o(n) \qquad \text{in probability}
\end{equation}

which implies in particular that 

\begin{equation}\label{eq:s to h}
\abs{\alpha_{ab}(\gamma^n,x)}_\infty=O(n) \qquad \text{in probability}.
\end{equation}



Now we use the triangle inequality 

\begin{align*}
d_\infty(\alpha(\gamma^n,x),n\overline{\alpha}_{ab}(\gamma))&\le d_\infty(\alpha(\gamma^n,x), \alpha_{ab}(\gamma^n,x))\\
&+d_\infty(\alpha_{ab}(\gamma^n,x), n\overline{\alpha_{ab}}(\gamma)). \\
\end{align*}

The second summand is $o(n)$ by \eqref{eq:2nd sum}. For the first summand, we apply Lemma \ref{l:oO(n)} with $h_n(x)=\alpha(\gamma^n,x)$; by \eqref{eq:s to h}, $\abs{\pi_{ab}h_n(x)}_\infty=O(n)$ in probability, while Lemma \ref{l:comm lambda} implies $\abs{\pi_{com}h_n(x)}_\infty=o(n)$ in probability. 
 
\end{proof}

\section{Asymptotic Behavior Along Arbitrary Elements}

In this section we prove the following. 

\begin{theorem}\label{T:4}
Let $\gamma_n\in\Gamma$  be a sequence satisfying

\[
\gamma_n=s_1^{a_{n,1}}\cdots s_k^{a_{n,k}}
\]

where $s_i\in S$ are fixed, in order, independent of $n$ and for each $i$, $\bbN\ni a_{n,i}\to\infty$ as $n\to\infty$. Then whp

\[
d_{\infty}(\alpha(\gamma_n,x),\delta_{a_{n,1}}\overline{\alpha_{ab}}(s_1)\star \cdots \star\delta_{a_{n,k}}\overline{\alpha_{ab}}(s_k))=o(\max a_{n,i}).
\]
\end{theorem}

We note that, for any sequence $\gamma_n\in\Gamma$, it is possible to write the $\gamma_n$ to satisfy the hypotheses of Theorem \ref{T:4}. Indeed, by Proposition 3.3 in \cite{Austin}, there is always a $K$ so that every $\gamma=s_1^{a_1}\cdots s_k^{a_k}$ with $a_i\in\bbN$, $s\in S$ and $k\le K$. By increasing $K$, one may assume that every $\gamma$ is represented with the same ordered generating set. By increasing $K$ again, we ensure $a_{n,i}\to\infty$ as $n\to\infty$ for each $i$. Indeed, for every $\gamma$, look at $a=\max a_i$, and for each $1\le j \le k$ so that $a_j<a/2$, rewrite $s_j^{a_j}=s_j^as_j^{a_j-a}$. We will not use either of these observations.

There is a natural way to compare the two points above. Using the cocycle equation we write

\[
\alpha(\gamma_n,x)=\alpha(s_1^{a_{n,1}},x_1)\cdots \alpha(s_k^{a_{n,k}},x_k)
\]

where $x_i:=s_{i+1}^{a_{n,{i+1}}}\cdots s_k^{a_{n,k}}x$. Proposition \ref{p: iterates H_infty} relates $\alpha(s_i^{a_{n,i}},x_i)$ and $\delta_{a_{n,i}}\overline{\alpha_{ab}}(s_i)$. We use the uniform boundedness of $k$ and Lemma \ref{L:multiplication} to extend Proposition \ref{p: iterates H_infty} to Theorem \ref{T:4}.
\begin{proof}[Proof of Theorem \ref{T:4}]
By the cocycle equation, it is enough to show that whp
\[
d_\infty(\alpha(s_1^{a_1},x_1)\cdots \alpha(s_k^{a_k},x_k),\delta_{a_{n,1}}\overline{\alpha_{ab}}(s_1)\star \cdots \star\delta_{a_{n,k}}\overline{\alpha_{ab}}(s_k))=o(\max a_{n,i}).
\]

For each $n$, let $a_n=\max a_{n,i}$. Now suppose the conclusion is false. Then there are $\epsilon,\delta>0$ and a subsequence (we keep the index $n$) so that
\[
m(x:d_\infty(\alpha(s_1^{a_{n,1}},x_1)\cdots \alpha(s_k^{a_{n,k}},x_k),\delta_{a_{n,1}}\overline{\alpha_{ab}}(s_1)\star \cdots \star\delta_{a_{n,k}}\overline{\alpha_{ab}}(s_k))>\delta a_n)>\epsilon
\]

Notice that $0\le a_{n,i}/a_n\le 1$. Therefore, after taking a diagonal subsequence, we may assume that $a_{n,i}/a_n\to a_i$ for each $1\le i \le k$. Proposition \ref{p: iterates H_infty} implies that, for every $1\le i \le k$, whp as $n\to\infty$
\[
\ssc{a_{n,i}}{\alpha(s_i^{a_{n,i}},x)}\longrightarrow \overline{\alpha_{ab}}(s_i).
\]

The above, and an easy calculation in coordinates using the definition of the $\delta_t$ and that $a_{n,i}/a_n\to a_i$ shows that for all $1\le i \le k$, whp as $n\to\infty$.
\[
\ssc{a_n}{\alpha(s_i^{a_{n,i}},x)}=\delta_{a_{n,i}/a_n}\ssc{a_{n,i}}{\alpha(s_i^{a_{n,i}},x)}\longrightarrow \delta_{a_i}\overline{\alpha_{ab}}(s_i).
\]

Invoking Lemma \ref{L:multiplication}, whp as $n\to\infty$
\[
\ssc{a_n}{\alpha(s_1^{a_{n,1}},x_1)\cdots \alpha(s_k^{a_{n,k}},x_k)}\longrightarrow \delta_{a_1}\overline{\alpha_{ab}}(s_1)\star\cdots\star\delta_{a_k}\overline{\alpha_{ab}}(s_k)
\]

which is equivalent to
\[
d_\infty(\alpha(s_1^{a_{n,1}},x_1)\cdots \alpha(s_k^{a_{n,k}},x_k),\delta_{a_na_1}\overline{\alpha_{ab}}(s_1)\star\cdots\star\delta_{a_na_k}\overline{\alpha_{ab}}(s_k))=o(a_n).
\]

But 
\[
d_\infty(\delta_{a_na_1}\overline{\alpha_{ab}}(s_1)\star\cdots\star\delta_{a_na_k}\overline{\alpha_{ab}}(s_k),\delta_{a_{n,1}}\overline{\alpha_{ab}}(s_1)\star\cdots\star\delta_{a_{n,k}}\overline{\alpha_{ab}}(s_k))=o(a_n),
\]

so we have a contradiction.

\end{proof}

\section{Construction of $\Phi$ and Proof of Main Theorem}
In this section we construct $\Phi$, prove Theorem \ref{T:Main} and deduce Theorem \ref{t:derivative}. 

\begin{defn}
Let $(G_\infty,\delta_t)$ be a graded nilpotent lie group with its one-parameter family of automorphisms. A finite symmetric subset $S\subset G_\infty$ \textit{generates} $G_\infty$ with respect to $\delta_t$, $t\ge 0$, if for every $g\in G_\infty$ there exist $k\in\bbN$, $s_1,\ldots ,s_k\in S$ and $a_1,\ldots ,a_k\in \bbR_+$ so that

\begin{equation}\label{eq:g rep}
g=\delta_{a_1}s_1\star\cdots\star \delta_{a_k}s_k.
\end{equation}
\end{defn}

\begin{example}\label{ex:genset}
In the Mal'cev coordinates on $G_\infty$, the set of $d'=\dim(G_\infty/G_\infty^2)$ elements
\[
\left \{ {(1,0,\ldots ,0),(0,1,0,\ldots 0),\ldots ,(0,\ldots ,1,0,\ldots ,0)}\right \}
\]

together with their inverses form a finite symmetric generating set for $G_\infty$ with respect to the homotheties $\delta_t$.

More generally any finite symmetric set with real span containing $V_1=\mfg/\mfg^2$ generates $G_\infty$ with respect to $\delta_t$. Indeed, the group generated by $\exp_\infty(V_1)$ is a connected subgroup of $G_\infty$, so by the Lie correspondence, its Lie algebra is a sub algebra of $\mfg_\infty$ containing $V_1$. Since $V_1$ generates $\mfg_\infty$ as a Lie algebra, the group generated by $\exp_\infty(V_1)$ is all of $G_\infty$.
\end{example}

We can now give a definition of $\Phi$ that will a priori depend on a choice of representation of $g\in G_\infty$ in the generating set $S$. Later on we will prove that there was in fact no choice involved. Let $S\subset G_\infty$ be the set of $2d'$ elements from Example \ref{ex:genset}. 

\begin{defn}[First Definition of $\Phi$]
\[
\Phi(g)=\delta_{a_1}\overline{\alpha_{ab}}(s_1)\star  \delta_{a_2}\overline{\alpha_{ab}}(s_2)\star\cdots\star \delta_{a_k}\overline{\alpha_{ab}}(s_k)
\]

where 

\[
g=\delta_{a_1}s_1\star \cdots\star \delta_{a_k}s_k
\]

is a fixed choice of a representation of $g$ as in \eqref{eq:g rep}.
\end{defn}

\begin{proposition}\label{p:one sequence} For each $g\in G_\infty$ there is a sequence $\gamma_n\in\Gamma$ so that 

\begin{itemize}
\item $\ssc{n}{\gamma_n}\to g$ \\
\item  $\ssc{n}{\alpha(\gamma_n,x)}\to \Phi(g)$ with high probability as $n\to\infty$.
\end{itemize}
\end{proposition}

\begin{proof}
Fix $g\in G_\infty$ and the choice of representation of $g$

\[
g=\delta_{a_1}s_1\star\cdots\star \delta_{a_k}s_k
\]

as in \eqref{eq:g rep}. For each $n\in\bbN$ and each $1\le i \le k$ set $m_{n,i}=\left \lfloor{na_i}\right \rfloor$, the greatest integer less than or equal to $na_i$. Then for each $1\le i \le k$ as $n\to\infty$ 

\begin{equation}\label{eq:conv to a_i}
\frac{m_{n,i}}{n}\to a_i.
\end{equation}

Now define for $n\in\bbN$

\[
\gamma_n=s_1^{m_{n,1}}s_2^{m_{n,2}}\cdots s_k^{m_{n,k}}.
\]

First notice that for each $1\le i \le k$ we have

\[
\ssc{n}{s_i^{m_{n,i}}}\to \delta_{a_i}s_i.
\]

Therefore by Lemma \ref{L:multiplication}

\[
\ssc{n}{\gamma_n}\to g,
\]

giving the first item. For the second item we invoke Theorem \ref{T:4}, which says that whp

\[
d_\infty(\alpha(\gamma_n,x),\delta_{m_{n,1}}\overline{\alpha_{ab}}(s_1)\star\cdots\star \delta_{m_{n,k}}\overline{\alpha_{ab}}(s_k))=o(\max m_{n,i}).
\]

By \eqref{eq:conv to a_i} the right hand side is $o(n)$. Thus whp as $n\to\infty$

\[
d_\infty(\delta_{1/n}\alpha(\gamma_n,x),\delta_{m_{n,1}/n}\overline{\alpha_{ab}}(s_1)\star\cdots\star \delta_{m_{n,k}/n}\overline{\alpha_{ab}}(s_k)) \to 0
\]

But as $n\to\infty$

\[
\delta_{m_{n,1}/n}\overline{\alpha_{ab}}(s_1)\star\cdots\star \delta_{m_{n,k}/n}\overline{\alpha_{ab}}(s_k)\to \Phi(g)
\]

which finishes the proof.
\end{proof}

The next Proposition says that $\ssc{n}{\alpha(\sigma_n,x)}\to \Phi(g)$ \emph{uniformly} as $\ssc{n}{\sigma_n}\to g$.

\begin{proposition}\label{p:all sequences}
Fix $g\in G_\infty$. For all $\epsilon_1,\epsilon_2>0$ there exist $\delta>0$ and $N\in\bbN$ so that whenever $\sigma\in\Gamma$ and $n\ge N$ are such that $d_{G_\infty}(\ssc{n}{\sigma},g)<\delta$, then with probability at least $1-\epsilon_1$ we have

\[
d_{H_\infty}(\ssc{n}{\alpha(\sigma,x)},\Phi(g))<\epsilon_2.
\]

In particular, for \emph{any} sequence $\ssc{n}{\sigma_n}\to g$ we have $\ssc{n}{\alpha(\sigma_n,x)}\to\Phi(g)$ in probability.
\end{proposition}

\begin{proof}
Fix $g\in G_\infty$ and $\epsilon_1,\epsilon_2>0$. Choose $\delta>0$ small so that $\kappa(1+\epsilon_2)2\delta/\epsilon_1<\epsilon_2$ where $\kappa=\max_{s\in S} \|\abs{\alpha(s,\cdot)}_\Lambda\|_1$. Let $\gamma_n$ be the sequence from Proposition \ref{p:one sequence}. Choose $N$ large so that for all $n\ge N$, $d_{G_\infty}(\ssc{n}{\gamma_n},g)<\delta$ and so that $d_{H_\infty}(\ssc{n}{\alpha(\gamma_n,x)},\Phi(g))<\epsilon_2$ with probability at least $1-\epsilon_1$. Choose $N$ larger if necessary so that the maps $\text{scl}_n^{\Gamma}$ and $\text{scl}_n^{\Lambda}$ are $(1+\epsilon_2)$-bi-Lipschitz for all $n\ge N$. 

Now suppose $d_{G_\infty}(\ssc{n}{\sigma},g)<\delta$ where $n\ge N$. Then $d_{G_\infty}(\ssc{n}{\sigma},\ssc{n}{\gamma_n})<2\delta$, which implies $d_\Gamma(\sigma,\gamma_n)<(1+\epsilon_2)n2\delta$. Set $\tau=\sigma^{-1}\gamma_n$, so $\abs{\tau}<(1+\epsilon_2)n2\delta$. By Markov's inequality

\[
m(\abs{\alpha(\tau,x)}_\Lambda \ge \kappa \abs{\tau}/\epsilon_1)\le\epsilon_1
\]

Thus by our choice of $\delta$, with probability at least $1-\epsilon_1$, we have

\[
\abs{\alpha(\tau,x)}\le n\epsilon_2.
\]

Using the cocycle equation $\alpha(\gamma_n,x)=\alpha(\sigma,\tau x)\alpha(\tau,x)$ and that $\text{scl}_n^{\Lambda}$ is $(1+\epsilon_2)$-bi-Lipschitz we have

\[
d_{H_\infty}(\ssc{n}{\alpha(\gamma_n,x)},\ssc{n}{\alpha(\sigma,\tau x)})<(1+\epsilon_2)\epsilon_2
\]

with probability at least $1-\epsilon_1$. Since $d_{H_\infty}(\ssc{n}{\alpha(\gamma_n,x)},\Phi(g))<\epsilon_2$ with probability at least $1-\epsilon_1$, we are done.
\end{proof}

The next Corollary says that the definition of $\Phi$ is independent of the choice of representation of $g$ in the generating set $S$.

\begin{corollary}
Suppose $g\in G_\infty$ can be written 

\[
g=\delta_{a'_1}s'_1\star\cdots\star \delta_{a'_{k'}}s'_{k'}.
\]

where $a'_i\in\bbR_+$ and $s'_i\in S$. Define

\[
\Phi'(g)=\delta_{a'_1}\overline{\alpha_{ab}}(s'_1)\star\cdots\star\delta_{a'_{k'}}\overline{\alpha_{ab}}(s'_{k'}).
\]

Then $\Phi(g)=\Phi'(g)$.
\end{corollary}

\begin{proof}
Repeat the proof of Proposition \ref{p:one sequence} with $\Phi'$ in place of $\Phi$. Doing so we obtain $\gamma'_n\in\Gamma$ so that $\ssc{n}{\gamma'_n}\to g$ and so that $\delta_{1/n}\alpha(\gamma'_n,x)\to \Phi'(g)$ in probability. By Proposition \ref{p:all sequences} $\delta_{1/n}\alpha(\gamma'_n,x)\to\Phi(g)$ in probability. Therefore $\Phi'(g)=\Phi(g)$.
\end{proof}

\subsection{$\Phi$ is a bi-Lipschitz group automorphism}

We can now show that $\Phi$ is a group isomorphism. Since any two Carnot-Carath\'eodory metrics on the same Carnot group are bi-Lipschitz to one another, we deduce that $\Phi$ is bi-Lipschitz. Let $\Psi$ denote the result of the above construction applied to the cocycle $\beta$ instead of $\alpha$. By symmetry, all of the results above apply equally to $\Psi$. We will see that $\Psi$ and $\Phi$ are inverses.

\begin{proposition}
$\Phi$ is a homomorphism.
\end{proposition}
\begin{proof}
Fix $g,h\in G_\infty$ and $\ssc{n}{\gamma_n}\to g$ and $\ssc{n}{\sigma_n}\to h$. Then by Lemma \ref{L:multiplication} and Proposition \ref{p:all sequences}
\begin{itemize}
\item $\ssc{n}{\alpha(\gamma_n,\sigma_nx)}\to\Phi(g)$ in probability \\
\item $\ssc{n}{\alpha(\sigma_n,x)}\to\Phi(h)$ in probability\\
\item $\ssc{n}{\gamma_n\sigma_n}\to gh$ \\
\item $\ssc{n}{\alpha(\gamma_n\sigma_n,x)}\to\Phi(gh)$ in probability.
\end{itemize}

Invoking Lemma \ref{L:multiplication} in $\Lambda$, with high probability

\[
 \ssc{n}{\alpha(\gamma_n,\sigma_nx)\alpha(\sigma_n,x)}\to \Phi(g)\star\Phi(h).
\]

Combining this with the fourth item, the proof is complete.
\end{proof}

To show that $\Phi$ and $\Psi$ are inverse maps, we need the following nilpotent group variant of Poincar\'e recurrence.

\begin{lemma}[Poincar\'e recurrence for nilpotent groups]\label{l:Poincare}
Fix $g\in G_\infty$ and let $A\subset X$ with $m(A)>0$. Then

\[
m\{ x\in A: \exists \gamma_{n_k} \in\Gamma~ \exists n_k\in\bbN ~\text{such that} \ssc{n_k}{\gamma_{n_k}}\to g~\text{and}~\gamma_{n_k}\cdot x\in A\}=m(A).
\]

\end{lemma}

\begin{lemma}\label{l:weak Poincare}
Fix $g\in G_\infty$, let $A\subset X$ with $m(A)>0$ and let $\delta>0$. Then

\[
m\{x\in A: \exists \gamma\in\Gamma~ \exists n\in\bbN ~\text{such that}~d_\infty(\ssc{n}{\gamma},g)<\delta ~\text{and}~ \gamma\cdot x\in A\}=m(A).
\]
\end{lemma}

\begin{proof}[Proof that Lemma \ref{l:weak Poincare} implies Lemma \ref{l:Poincare}]
Set

\[
A_\delta=\{x\in A: \exists \gamma\in\Gamma~ \exists n\in\bbN ~\text{such that}~d_\infty(\ssc{n}{\gamma},g)<\delta ~\text{and}~ \gamma\cdot x\in A\}
\]

which has measure $m(A)$ by \ref{l:weak Poincare}. Then

\[
A'=\cap_{l=1}^{\infty} A_{1/l}
\]

again has measure $m(A)$, and has the desired property.
\end{proof}

\begin{proof}[Proof of Lemma \ref{l:weak Poincare}]
Suppose for contradiction that there is $g\in G_\infty$, $A \subset X$ with $m(A)>0$, $\delta>0$ and $E\subset A$ with $m(E)>0$ such that for a.e. $x\in E$, $d_\infty(\ssc{n}{\gamma},g)<\delta$ implies $\gamma\cdot x\notin A$.

We claim that there exist infinitely many $(n_k,\gamma_{n_k})\in\bbN\times\Gamma$ such that \\ $d_\infty(\ssc{n_k}{\gamma_{n_k}},g)<\delta$ and such that if $k_i<k_j$ then 

\[
d_\infty(\ssc{n_{k_j}}{\gamma^{-1}_{n_{k_i}}\gamma_{n_{k_j}}},g)<\delta.
\]

Indeed, pick any $(n_1,\gamma_{n_1})$ so that $d_\infty (\ssc{n_1}{\gamma_{n_1}},g)<\delta$. Now consider any sequence $\ssc{m}{\gamma_m}\to g$. Since $\ssc{m}{\gamma^{-1}_{n_1}}\to id$ as $m\to\infty$, Lemma \ref{L:multiplication} implies that

\[
\ssc{m}{\gamma^{-1}_{n_1}\gamma_m}\to g.
\]

Thus we may pick $n_2:=m$ large to satisfy the claim. Continuing in this way, the claim is proved.

Now we see that the sets $\gamma_{n_k} E$ are pairwise disjoint: indeed, if not, then 

\[
m(\gamma^{-1}_{n_{k_i}}\gamma_{n_{k_j}} E \cap E)>0.
\]

which implies that there is a positive measure set of $x \in E$ so that $\gamma^{-1}_{n_{k_i}}\gamma_{n_{k_j}} \cdot x \in E\subset A$ while $d_\infty(\ssc{n_{k_j}}{\gamma^{-1}_{n_{k_i}}\gamma_{n_{k_j}}},g)<\delta$, contradicting the definition of $E$.

Thus the sets $\gamma_{n_k}E$ are pairwise disjoint. But as $m(E)>0$, this is also impossible.
\end{proof}

Notice that, while one can formulate the Lemmas \ref{l:Poincare} and \ref{l:weak Poincare} for any group together with one of its asymptotic cones, the key ingredient that fails for groups that are not nilpotent is Lemma \ref{L:multiplication}. This is easily seen in the free group.

\begin{proposition}
$\Phi$ and $\Psi$ are inverse maps. Consequently, they are group isomorphisms.
\end{proposition}
Recall (\S \ref{ss:ime}) that the fundamental domains $X$ and $Y$ satisfy $m(X\cap Y)>0$ and that $x\in X\cap Y \cap \gamma^{-1}(X\cap Y)$ implies that $\beta(\alpha(\gamma,x),x)=\gamma$. 

\begin{proof}
Fix $g\in G_\infty$ and $\epsilon>0$. We will show that $d_\infty(\Psi(\Phi(g)),g)<2\epsilon$. Using the symmetry of $\alpha$ and $\beta$, we apply Proposition \ref{p:all sequences} to the cocycle $\beta$, the map $\Psi$ and the element $\Phi(g)$ to obtain $N\in\bbN$ and $\delta>0$ so that for any $\gamma_n\in\Gamma$ with $n\ge N$ and any $x\in X$, for a positive measure set of $y\in X\cap Y$

\begin{equation}\label{eq:beta alpha}
d_{H_\infty}(\ssc{n}{\alpha(\gamma_n,x)},\Phi(g))<\delta \implies d_{G_\infty}(\ssc{n}{\beta(\alpha(\gamma_n,x),y)}.\Psi(\Phi(g)))<\epsilon.
\end{equation}

Now applying Proposition \ref{p:all sequences} to $\alpha$, $\Phi$ and $g$ we obtain $\delta'>0$ and $N'\in\bbN$ so that whenever $n\ge N$

\[
d_{G_\infty}(\ssc{n}{\gamma_n},g)<\delta'
\]

implies that for a positive measure subset of $X \cap Y$ both \eqref{eq:beta alpha} occurs and

\[
d_{H_\infty}(\ssc{n}{\alpha(\gamma_n,x)},\Phi(g))<\delta.
\]

Choose $\delta'<\epsilon$ if necessary, and set $N=\max(N,N')$. Then with positive probability in $X\cap Y$, for $n\ge N$

\begin{equation}\label{eq:XY}
d_{G_\infty}(\ssc{n}{\gamma_n},g)<\delta' \implies d_{G_\infty}(\ssc{n}{\beta(\alpha(\gamma_n,x),x)},\Psi(\Phi(g)))<\epsilon.
\end{equation}

Now we invoke Lemma \ref{l:Poincare} (Poincar\'e Recurrence) applied to $X\cap Y$, $g$ and $\delta'$ to assert that with positive probability in $X\cap Y$ there exists $n\ge N$ and $\gamma_n\in\Gamma$ with $d_{G_\infty}(\ssc{n}{\gamma_n},g)<\delta'$, such that $\gamma_n x\in X\cap Y$ and such that \eqref{eq:XY} occurs. Therefore with positive probability 

\[
d_{G_\infty}(\ssc{n}{\gamma_n},\Psi(\Phi(g)))<\epsilon \qquad \text{and} \qquad d_{G_\infty}(\ssc{n}{\gamma_n},g)<\epsilon.
\]
\end{proof}

\subsection{Theorem \ref{T:Main} implies Theorem \ref{t:derivative}}

\begin{proof}
We recall the definition of the maps $\kappa_{x,n}$.  For each $n\in\bbN$ the maps $\text{scl}^{G_\infty}_n(-):\Gamma\to G_\infty$ map $\Gamma$ more and more densely into $G_\infty$ and similarly for  $\text{scl}^{H_\infty}_n(-):\Lambda\to H_\infty$ (see  \S 2). For every $g\in G_\infty$ and every $n\in\bbN$ let $j_n(g)\in \Gamma$ be an element of $\Gamma$ minimizing the distance between $\text{scl}_n^{G_\infty}(\Gamma)$ and $g$. Then for $g\in G_\infty$ we define

\[
\kappa_{x,n}(g)=\text{scl}_n^{H_\infty}(\alpha(j_n(g),x)).
\]

Now fix $R>0$, $\delta>0$ and $\epsilon>0$. Let $B_R^{G_\infty}(e)$ denote the ball of radius $R>0$ in $(G_\infty,d_\infty)$ about the identity. By Theorem \ref{T:Main}, for every $g\in G_\infty$ there is $\tau=\tau(g)>0$ so that whenever $\text{scl}_n^{G_\infty}(\gamma_n) \in B_{\tau}^{G_\infty}(e)$, with probability at least $1-\delta$ we have 

\[
d_{H_\infty}(\Phi(g),\text{scl}_n^{H_\infty}(\alpha(\gamma_n,x)))<\epsilon.
\]

By the compactness of $B_R^{G_\infty}(e)$ we obtain a finite set $F\subset B_R^{G_\infty}(e)$ with the property that for every $g\in B_R^{G_\infty}(e)$ there is $g_0\in F$ so that $d_{G_\infty}(g,g_0)<\epsilon$ and so that 

\[
g\in B_{\tau(g_0)/2}^{G_\infty}(g_0).
\]

Now set $\tau=\min_F \tau(g)$ and choose $N$ large so that for all $n\ge N$, for all $g\in B_R^{G_\infty}(e)$ we have

\[
d_{G_\infty}(\text{scl}_n^{G_\infty}(j_n(g)),g)<\tau/2.
\]

Then for all $n\ge N$ and every $g\in B_R^{G_\infty}(e)$ there is $g_0\in F$ so that with probability at least $1-\delta$

\[
d_{H_\infty}(\Phi(g_0),\text{scl}_n^{H_\infty}(\alpha(j_n(g),x)))<\epsilon
\]

and 

\[
d_{G_\infty}(\Phi(g),\Phi(g_0))<L\epsilon
\]

where $L$ is the Lipschitz constant for $\Phi$. This finishes the proof.
\end{proof}

\end{document}